\documentclass[12pt,twoside]{article}
\usepackage{amsmath}
\usepackage{amssymb}
\usepackage{amsthm}
\usepackage{url}

\DeclareMathOperator{\R}{\mathfrak{R}}
\DeclareMathOperator{\riem}{rm}

\DeclareMathOperator{\ric}{rc}

\DeclareMathOperator{\scal}{sc}
\DeclareMathOperator{\dd}{\,d}

\DeclareMathOperator{\id}{I}

\DeclareMathOperator{\dist}{dist}

\DeclareMathOperator{\vol}{Vol}
\DeclareMathOperator{\inj}{inj}

\newcommand{\D}{\nabla}

\newcommand{\dv}{\,\mathrm{dvol}}

\newcommand{\td}[1]{\frac{\mathrm{d}}{\mathrm{d}#1}}

\newcommand{\pd}[1]{\frac{\partial}{\partial #1}}

\newcommand{\abs}[1]{\lvert{#1}\rvert}
\newcommand{\norm}[1]{\|{#1}\|}

\newcommand{\e}[1]{\mathrm{e}^{#1}\,}
\newcommand{\ee}{\mathrm{e}}

\newcommand{\pbar}{\bar{p}}

\newcommand{\tbar}{\bar{t}}
\newcommand{\gtil}{\tilde{g}}

\newcommand{\vbh}[1]{\vol_{\mathrm{H}}(#1)}

\newcommand{\rn}{\mathbb{R}}

\theoremstyle{definition} 
\newtheorem{definition}{Definition}[section]

\theoremstyle{plain}      
\newtheorem{thm}[definition]{Theorem}
\newtheorem{lemma}[definition]{Lemma}
\newtheorem{cor}[definition]{Corollary}
\newtheorem{prop}[definition]{Proposition}
\newtheorem{claim}[definition]{Claim}
\theoremstyle{remark}
\newtheorem{rmk}[definition]{Remark}

\title{On Uniqueness of Complete Ricci Flow Solution with
  Curvature Bounded from Below }
\author{Li Sheng\footnote{Department of Mathematics, Sichuan University,
    Chengdu, 610064, China. Email: \emph{lshengscu@gmail.com}}
  , Xiaojie Wang\footnote{Mathematics Department, Stony Brook
    University, Stony Brook, NY 11794-3660, USA Email:
    \emph{wang@math.sunysb.edu}}}

\begin{document}
 
\maketitle

\begin{abstract}
  Let $(M,g)$ be a complete noncompact non-collapsing
  $n$-dimensional riemannian 
  manifold, whose complex sectional curvature is bounded from below
  and scalar curvature is bounded from above. Then ricci
  flow with above 
  as its initial data, has at most one solution in the class of
  complete riemannian metric with complex sectional curvature bounded
  from below.  
\end{abstract}

\section{Introduction}
\label{sec:introduction}

Let $M$ be a differentiable manifold, $I$ be an interval, and
$g(t)\equiv g_t\equiv g(x,t)$, 
$(x,t)\in M\times I$, be
a family of complete riemannian metrics on $M$, which is parameterized
by $t\in I$. Let $\ric_t$ be ricci curvature of $g_t$. 
The ricci flow 
is a weakly parabolic system of partial differential equations defined by
\begin{equation*}
\pd{t}g_t=-2\ric_t.
\end{equation*}
It was first
introduced by Richard Hamilton in his celebrated 1982
paper\cite{0504.53034} and used to deform riemannian metrics.
 Thereafter, as a powerful tool,
ricci flow helped changing the landscape of differential
geometry greatly. For example, it is a crucial tool to prove Poincar\'e
conjecture and differentiable sphere theorem. See, for example,
 \cite{2003math......3109P} and \cite{MR2449060}.  

Ricci flow as a system of partial differential equations, the short time
existence and uniqueness of the solution to initial value problem
are among the most fundamental questions. When underlying manifold
is compact without boundary, they were proved by Hamilton 
in the same paper mentioned
above.  Shortly after, DeTurck\cite{MR697987} introduced a
technique now commonly named after himself, to 
simplify the proof of existence. Together with the help from 
harmonic heat flow, Hamilton used DeTurck's trick to give a new proof to the
uniqueness in \cite{MR1375255}. 

For the case of compact manifold with boundary, which we will not discuss
in this paper, interested readers may consult
\cite{2012arXiv1210.0813G} and bibliography  therein. 

When the manifold is complete non-compact without boundary, Shi
\cite{MR1001277} first constructed a complete
solution to every initial metric with bounded riemannian curvature. If
the curvature bound is $k_0>0$, then  there 
exists a $T=T(n,k_0)$, such that his solution exists on $[0,T]$ and satisfies
\begin{equation*}
  \abs{\D^m \riem}^2\leq \frac{c_m}{t^n}
\end{equation*}
for some constant $c_m>0$.
Cabezas-Rivas and Wilking\cite{2011arXiv1107.0606C} improved Shi's result 
by removing the curvature upper bound for the initial data, and only
assuming that it is non-collapsing and has nonnegative complex
sectional curvature (see below for definition). Furthermore, their
solution preserves
nonnegativity of complex sectional curvature. 

For the uniqueness of solution to classic heat equation,  in contrast,
one cannot 
take  the uniqueness of solution for granted. It is Tychonoff who had not only
discovered a non-zero 
solution to heat equation on $\rn\times [0,\infty)$ with vanishing initial
data, but also proved uniqueness by assuming solutions at most exponential
growth (see e.g. John\cite{MR831655} page 211 and 217). However, in
the case of ricci flow, the situation is revealed to 
be subtle by the discovery that, in dimension two and three, the only
solution to ricci flow with 
euclidean initial data is euclidean. We will discuss this later.

Nonetheless, following spirit of Tychonoff, Chen and Zhu established a
uniqueness theorem in the 
class of metrics with bounded curvature.
\begin{thm}[Chen-Zhu \cite{MR2260930}]
  \label{thm:4}
  Let $(M^n,g)$ be a complete noncompact riemannian manifold of
  dimension $n$ with bounded curvature. Let $g_1(t)$ and $g_2(t)$
  be two solutions to the ricci flow on $M\times [0,T]$ with $g$ as
  their common
  initial data and with bounded curvatures. Then $g_1=g_2$ 
  for all $(x,t)\in M^n\times [0,T]$.
\end{thm}



Also along this spirit, Kotschwar \cite{2012arXiv1206.3225K} weakened
the condition to  curvature growth rate 
being at most quadratic and equivalence of $g(t)$ to $g(0)$.



In lower dimensions, however, the ricci flow is ``better''
than the classic heat equation. 
Topping\cite{MR2832165}\cite{2013arXiv1305.1905T}  proved that for
any initial  
metric on  two dimensional manifold, which could be noncomplete
and/or with unbounded curvature, ricci flow has one and only one
instantaneously complete solution. Readers may also see
\cite{MR2520796} for related discussion on surface.

In dimension three, Chen\cite{MR2520796} proved that for any
complete noncompact  initial metric, if it has nonnegative and bounded
sectional curvature, then on some time interval depending only on
initial curvature bound, there exists a unique complete ricci flow
solution. As a corollary, the only complete ricci flow solution on $\rn^3$
with euclidean initial metric is euclidean. His idea of the proof is first to 
prove that the curvature is bounded
over some time interval, then the uniqueness follows from the theorem
of Chen-Zhu.  

Another way to weaken the condition is
replacing the bilateral curvature bound by unilateral. We first
consider the case with curvature upper bound.
On riemann surface, Chen and Yan  \cite{MR2721664}
proved that curvature upper bound is sufficient. It is worth noting
that their
theorem does not require curvature lower bound for the initial metric.

In general dimension, when we assume in addition that the initial metric
is with bounded
curvature, one can prove the uniqueness by assuming only
the upper bound of sectional curvature. In fact, this is a
straight forward consequence of Corollary 2.3 in \cite{MR2520796},
which says when scalar curvature is bounded from below initially, it
is preserved along the flow.  
So we can apply theorem \ref{thm:4} to obtain uniqueness.





Then we turn to the case with curvature lower bound, which is the main
purpose of this paper.
Before stating our main result, we first remind
readers of the definition of complex sectional curvature. 

Let $(M^n,g)$ be a
$n$-dimensional riemannian manifold  and
$T^cM=TM\otimes\mathbb{C}$ be complexified tangent bundle.
Let $\D$ be Levi-Civita connection, and $\riem$ be riemannian
curvature defined by 
\begin{equation*}
  \riem(u,v,w,z)=g(\D_u\D_vw-\D_v\D_uw-\D_{[u,v]}w,z)
\end{equation*}
where $u,v,w,z\in T_pM$ for any $p\in M$.  And 
\begin{equation*}
  \R:\wedge^2TM\mapsto\wedge^2TM
\end{equation*}
is the riemannian curvature operator induced by $\riem$. Let $\scal$
be the scalar curvature.
We extend both metric $g$ and
curvature tensor $\riem$ to be $\mathbb{C}$-multilinear maps
\begin{equation*}
  g:S^2T^cM\to \mathbb{C}\quad
  \riem:S^2(\wedge^2T^cM)\to \mathbb{C}.
\end{equation*}
\begin{definition}
  \label{def:1}
For any $2$-dimensional complex subspace $\sigma\in
T_p^cM$, the complex sectional curvature 
is defined by
\begin{equation*}
   \sec^c(\sigma):=\riem(u,v,\bar{v},\bar{u})=g(\R(u\wedge
   v),\overline{u\wedge v}),
\end{equation*}
for any unitary orthogonal vectors $u,v\in \sigma$ 
i.e. $g(u,\bar{u})=g(v,\bar{v})=1$ and $g(u,\bar{v})=0$, where, e.g. 
$\bar{u}$ is the complex conjugate of $u$. 
\end{definition}

We can now state the main result. 
\begin{thm}
  \label{thm:1}
  Let $(M^n,g)$ be a complete noncompact $n$-dimensional riemannian
  manifold, with complex sectional curvature $\sec^c \geq -1$
  and $\scal\leq s_0$ for some $s_0 > 0$. In addition, the volume of every unit
  ball is bounded from below uniformly by some constant
  $v_0>0$.  Let $g_i(t)$ $i =  1,2$ be 
  two solutions to the ricci flow  on
  $M\times [0,T]$ with same initial data $g_i(0) = g$.  If both
  $g_i(t)$ are
  complete riemannian metrics with $\sec^c(g_i(t))\geq -1$, for all $t\in [0,T]$,
  then there exists a constant $C_n>0$ depending only on dimension $n$,
  s.t. $g_1(t) = g_2(t)$ on $ M\times [0,\min\{\epsilon,T\}]$, where $\epsilon=\min\{v_0/2C_n,1/(n-1)\}$.   
\end{thm}
We need the volume lower bound for unit ball of initial metric to obtain
injectivity radius lower bound of $g(t)$ when $t>0$. This is crucial to
establish Cheeger-Gromov convergence of ricci flows in our proof. Actually,
the curvature condition alone cannot guarantee the volume lower
bound, which is shown by the following example. Consider a
rotationally 
symmetric metric $g=dr^2 + \e{-2r}ds^2$ on a half-cylinder. Its
curvature is $-\frac{(\e{-r})''}{\e{-r}}=-1$. Closing up one end at $r=0$ by
putting on a cap, smoothing if necessary, we get a complete metric
with bounded curvature.  But the volume of unit ball 
goes to zero when $r$ increases to infinity. By crossing with $\rn^n$ we
get examples for higher dimensions.  

With Theorem \ref{thm:4}, above theorem is an immediate
corollary of the following one.
\begin{thm}
  \label{thm:3}
  Let $(M^n,g)$ be as in above theorem. Let $g(t)$ be a ricci flow solution
  on $M^n\times [0,T]$ with $g(0)=g$. If $\sec^c_t\geq -1$  
  and $g(t)$ is complete for any  $t\in [0,T]$,
  then $\riem_t$ is uniformly bounded on
  $M\times [0,\min\{\epsilon,T\}]$ for the same $\epsilon$ in above theorem.
\end{thm}
\begin{rmk}
  \label{rmk:5}
  Readers may wonder if the uniqueness is true on
  $[0,T]$. Unfortunately, our method is not strong enough to prove it. Actually,
  Cabezas-Rivas and Wilking had constructed an immortal complete
  ricci flow solution with positive curvature operator (hence with
  positive complex 
  sectional curvature) which is bounded if and only if $t\in
  [0,1)$ (See \cite{2011arXiv1107.0606C} Theorem 4 b). And by a
  Toponogov's theorem below (Theorem \ref{thm:2}), it has
  a uniform lower bound for volume of any unit ball in 
  initial manifold. So it does satisfy assumptions of
  Theorem \ref{thm:1}. However, our method to prove uniqueness is
  to show that curvature 
  is bounded, which is not true general for the whole interval of
  $[0,T]$ as shown by above example. Nonetheless, the question is still open.
\end{rmk}
\begin{rmk}
  \label{rmk:6}
  Authors are in debt to Cabezas-Rivas and 
  Wilking \cite{2011arXiv1107.0606C} for
  following their use of complex 
  sectional curvature and adapting
  many of their arguments to serve their own purpose. 
\end{rmk}

\section{Preliminaries}
\label{sec:preliminaries}

In this section, we list a few facts for later use. 

Along ricci flow, the volume form satisfies the following evolution equation.
\begin{equation*}
\pd{t}\dv = -\scal\dv  
\end{equation*}
We also write down the rescaling relations for later use. If  $\gtil=k^2g$,
$\widetilde{\sec}^c=k^{-2}\sec^c$.

Remember that  sectional curvature of a tangent plane $\sigma$ spanned by
$u,v\in T_pM$ is defined by
\begin{equation*}
  \sec(\sigma)=\frac{\riem(u,v,v,u)}{\abs{u}^2\abs{v}^2-g(u,v)^2}.
\end{equation*}
We have the following relations between complex sectional curvature and
other curvatures.
\begin{lemma}
  \label{lem:3}
  If $\sec^c>(\geq)C$ for any constant $C$, then sectional
  curvature $\sec >(\geq)C$.
\end{lemma}
\begin{proof}
  Let $u,v\in T_pM\subset T_p^c$ be any pair of orthogonal
  unit vectors and  
  $\sigma$ be the complex plane spanned by them. Then sectional
  curvature on the real plane spanned by them is
  $\sec(u,v)=\riem(u,v,v,u)=\sec^c(\sigma)>C$. 
\end{proof}
\begin{lemma}
  \label{lem:4}
  If sectional curvature is bounded from below and scalar curvature is
  bounded from above, then the sectional curvature is
  bounded from both sides, as is the curvature operator.
\end{lemma}
\begin{proof}
  Fix a point $p\in M$.
  Suppose for any plane $P\in T_pM$, the sectional curvature $\sec(P)\geq
  c$; and the scalar curvature  $\scal(p)\leq C$  for some fixed  
  constants $c\leq C$. 
  
  Then for any o.n. basis $\{e_i\}$ of  $T_pM$,
  \begin{equation*}
    C>\scal=\sum_{i,j}\riem(e_i,e_j,e_j,e_i)=2\sum_{i<j}\riem(e_i,e_j,e_j,e_i)=
    2\sum_{i<j}\sec(P_{ij})
  \end{equation*}
  where $P_{ij}$ is the plane spanned by $\{e_i,e_j\}$. Because
  $\sec$ is bounded from below, by above inequality, $\sec(P_{ij})$ is
  bounded from above by a constant depending on only $c$, $C$, and
  dimension. The above orthonormal basis is 
  arbitrary, so sectional curvature at $p$ is bounded from above. 

  Any component of curvature operator is a linear combination of
  sectional curvature for various planes in tangent space where the
  coefficients depend only on dimension (see  e.g. Cheeger and Ebin
  \cite{MR2394158} page 14). Thus curvature operator is bounded.  
\end{proof}

\section{Proof of Theorem \ref{thm:3}}
\label{sec:proof-theorem}

\begin{proof}[Proof of theorem~\ref{thm:3}]
  \label{pf:1}
  The key to the proof is establishing the following  estimate
  \begin{equation*}
    \scal_t\leq \frac C t.
  \end{equation*}
  Then, using Theorem~\ref{thm:5} in B.L.Chen \cite{MR2520796} (see
  also Simon \cite{MR2439551}), we
  get a uniform bound for $|\riem|$. 
  The arguments here are adaptations of those in Cabezas-Rivas and Wilking
  \cite{2011arXiv1107.0606C}.
  
  \subsection{Estimate Volume Lower Bound}
  \label{sec:estim-volume-lower}
  
  We first establish uniform volume lower bound for unit balls along
  ricci flow. 
  \begin{lemma}
    \label{lem:1}
    Let $(M,g)$ be a complete riemannian $n$-manifold whose unit balls have
    volume lower bound $v_0>0$.  Let $(M,g(t))$ be a solution of 
    ricci flow on $t\in [0,T]$ with $g(0)=g$,  and  its
    sectional curvature is 
    bounded from below by $-1$ for any $t\in [0,T]$, then there exists
    a constant
    $C_n > 0$ depending only on dimension $n$, s.t. The volume of any ball
    $B_t(p,\ee)\subset
    (M,g(t))$  bounded from below by $v_0/2$ for $t\in [0,\epsilon]$, 
    where $\epsilon=\min\{v_0/2C_n,1/(n-1)\}$.  
  \end{lemma}
  \begin{proof}[Proof of Lemma \ref{lem:1}]
    \label{pf:3}
    Because sectional curvature is bounded uniformly from below by $-1$,
    $\ric_t \geq -(n-1)g(t)$, we have the following estimate for
    distance function.
    \begin{lemma}
      \label{lem:5}
      For any fixed points $p,q\in M$, 
      \begin{equation}
        \label{eq:2}
        \dist_t(p,q)\leq \dist_0(p,q)\,\e{(n-1)t}.
      \end{equation}
    \end{lemma}
    \begin{proof}
      \label{pf:5}
      Let $\gamma:[0,1]\to M$ be a length minimizing geodesic with
      respect to $g(0)$ such that $\gamma(0)=p$ and $\gamma(1)=q$. Let
      $\abs{\gamma}_t$ be the length of the curve $\gamma$ with
      respect to $g(t)$. Then
      \begin{equation*}
        \begin{split}
          \td{t}\abs{\gamma}_t
          &=\int_0^1\td{t}\sqrt{g_t(\dot\gamma(s),\dot\gamma(s))}\dd s
          =\int_0^1
          \frac{-\ric(\dot\gamma(s),\dot\gamma(s))}
          {\sqrt{g_t(\dot\gamma(s),\dot\gamma(s))}}\dd s\\
          &\leq\int_0^1 (n-1)\norm{\dot\gamma(s)}_t\dd s=(n-1)\abs{\gamma}_t.
        \end{split}
      \end{equation*}
      By integrating above differential inequality, we get
      \begin{equation*}
        \dist_t(p,q)\leq\abs{\gamma}_t\leq \abs{\gamma}_0\e{(n-1)t}
        =\dist_0(p,q)\e{(n-1)t}.
      \end{equation*}
    \end{proof}
    Therefore $B_0(p,1)\subset B_t(p,\ee)$ for any $0\leq t\leq 1/(n-1)$ and
    any $p\in M$. Then,
    \begin{equation*}
      \begin{split}
        \td t \vol_t(B_0(p,1))&= -\int_{B_0(p,1)}\scal_t
        \dv_t\\
        &= -\int_{B_t(p,\ee)}\scal_t\dv_t +
        \int_{B_t(p,\ee)-B_0(p,1)}\scal_t\dv_t\\
        &\geq -C(n) + \int_{B_t(p,\ee)}-n(n-1)\dv_t.\\
      \end{split}
    \end{equation*}
    In last step we use $\scal_t\geq -n(n-1)$ and Petrunin's estimate as below.
    \begin{thm}[Petrunin\cite{MR2423998}]
      \label{thm:101}
      Let $M$ be a complete riemannian manifold whose sectional
      curvature is at least $-1$. Then
      \begin{equation*}
        \int_{B(p,1)}\scal \leq C(n)
      \end{equation*}
      for any $p\in M$, where $B(p,1)$ is the unit ball centered  in
      $p$ and $C(n)$ is a constant depending only on dimension $n$.
    \end{thm}
    Then, applying Bishop-Gromov volume comparison,
    \begin{equation*}
        \td t \vol_t(B_0(p,1))
        \geq -C(n) - n(n-1)\vbh \ee
        =:-C'(n),
    \end{equation*}
    where $\vbh r$ is the volume of a ball with radius $r$ in
    simply connected hyperbolic space of constant curvature $-1$.
    Integrating above,
    \begin{equation*}
      \vol_{t}(B_{t}(p,\ee))\geq 
      \vol_{t}(B_0(p,1))\geq \vol_0(B_0(p,1)) - C'(n) t\geq v_0 -C‘(n) t.
    \end{equation*}
    Hence, let $\epsilon=\min\{v_0/2C'(n),1/(n-1)\}$, 
    \begin{equation*}
      \vol_{t}(B_{t}(p,\ee)\geq v_0/2\ \text{for any}\ t\in
      [0,\epsilon]. 
    \end{equation*}
  \end{proof}

  \subsection{Curvature Uniform Bound}
  \label{sec:curv-unif-bound}
  
  \begin{lemma}
    \label{lem:2}
    Let $(M^n,g(t))$ be a complete solution of ricci flow
    for $M\times [0,\epsilon]$ with $\sec^c\geq -1$. If there
    is a constant $v'>0$ such that
    volume of balls $\vol_t(B_t(p,\ee)) > v'$, for any $(t,p)\in
    [0,\epsilon]\times M$,  then there exists a
    constant $C>0$ ,  so that
    \begin{equation*}
      \scal_t(p) \leq \frac {C} t
    \end{equation*}
    for any $(p,t)\in M\times (0,\epsilon]$.
  \end{lemma}
  \begin{proof}
    \label{pf:2}
    Prove by contradiction. Suppose there is a
    sequence  ${(p_k,t_k)}\subset M\times (0,\epsilon]$  such that, 
    \begin{equation}
      \label{eq:7}
      \scal_k(p_k) > \frac{4^k}{t_k}.
    \end{equation}
    where and hereafter, $B_k=B_{g(t_k)}$, $B_t=B_{g(t)}$, $\scal_t=\scal_{g(t)}$,
    $\scal_k=\scal_{g(t_k)}$, $\dist_t=\dist_{g(t)}$, and  $\dist_k=\dist_{g(t_k)}$.
    We can pick points in space-time to blow up by following trick first
    introduced by Perelman in \cite{2002math.....11159P}.
    \begin{claim}
      \label{claim:2}
      For any sufficiently large integer $k$, there exists
      $\pbar_k\in M$ and  $\tbar_k\in (0,t_k]$  satisfy
      the following equations
      \begin{equation}
        \label{eq:8}
        \scal_t(x) \leq 2 \scal_{\tbar_k}(\pbar_k)\quad\text{for all}\ 
        \begin{cases}
          x \in B_{\tbar_k}(\pbar_k,
          \frac{k}{\sqrt{\scal_{\tbar_k}(\pbar_k)}})\ \text{and,}\\ 
          t \in (\tbar_k-\frac{k}{\scal_{\tbar_k}(\pbar_k)},\tbar_k]\\
        \end{cases}
      \end{equation}
      \begin{equation}
        \label{eq:1}
        \scal_{\tbar_k}(\pbar_k) \geq \frac{4^k}{t_k}.
      \end{equation}
    \end{claim}
    \begin{proof}[Proof of Claim~\ref{claim:2}]
      \label{pf:4}
      We start searching for $(\pbar_k,\tbar_k)$ from $(p_k,t_k)$
      --- if it satisfies equation \eqref{eq:8},  we
      are done. Otherwise we can find $x_1$ and $\tau_1$ 
       s.t.
       \begin{equation*}
         \dist_k(x_1,p_k)\leq \frac{k}{\sqrt{\scal_k(p_k)}},\  
         \tau_1 \in (t_k-\frac{k}{\scal_k(p_k)},t_k],\ \textrm{and }
         \scal_{\tau_1}(x_1) \geq 2 \scal_k(p_k).
      \end{equation*}
      If $(x_1,\tau_1)$ is not what we are looking for, we can find
      $x_2$ and $\tau_2$, s.t.
      \begin{equation*}
        \dist_{\tau_1}(x_1,x_2)\leq\frac k{\sqrt{\scal_{\tau_1}(x_1)}},\  
        \tau_2\in (\tau_1-\frac k{\scal_{\tau_1}(x_1)},\tau_1], 
        \ \text{and } \scal_{\tau_2}(x_2) \geq 2 \scal_{\tau_1}(x_1).
      \end{equation*}
      We claim that, along this way only for finite steps, we can find
       desired 
      $(\tbar_k,\pbar_k)$.  Otherwise, we get a sequence of $(x_k,\tau_k)$ and
      $\scal_{\tau_k}(x_k)\to \infty$ as  $k\to\infty$. Here readers may
      worry that $\tau_i$ would go below $0$ thus ill-defined. However, by our
      construction, not only this cannot happen, but also, the above
      sequence is well bounded in both space and time, as
      shown in the following discussion. For the convenience, let
      $\tau_0:=t_k$ and $x_0:=p_k$.
      
      Firstly,
      \begin{equation*}
        \begin{split}
        t_k\geq \tau_{i+1}&\geq
        t_k - \sum_{l=0}^i\frac k {\scal_{\tau_l}(x_l)}\\
        &\geq t_k-t_k\sum_{l=0}^i\frac k {2^l\scal_k(p_k)} \geq
        t_k\left (1- \frac{2k}{\scal_k(p_k)}\right )
        =: \epsilon_k >0.
        \end{split}
      \end{equation*}
      for sufficiently large $k$, where $\epsilon_k$ is a constant
      independent of $i$.  

      Secondly, $x_i$ is also
      bounded with respect to a background metric $g_k$. Based on assumption,
      ricci curvature is bounded from below by $-(n-1)$,
      and by definition  $\tau_i < t_k$, then by
      equation~\eqref{eq:2}, for any integer $l\geq 1$
      $\dist_k(x_{l-1},x_l)\leq
      \e{(n-1)t_k}\dist_{\tau_{l-1}}(x_{l-1},x_l)$. 
      Thus
      \begin{equation*}
        \begin{split}
          \dist_k(p_k,x_i)&\leq
          \e{(n-1)t_k}\sum_{l=1}^i\dist_{\tau_{l-1}}(x_{l-1},x_l)\\
          &\leq \e{(n-1)t_k} \sum_{l=1}^i\frac k
          {\sqrt{\scal_{\tau_{l-1}}(x_{l-1})}} \\
          &\leq \e{(n-1)t_k} \sum_{l=1}^\infty\frac k
          {(\sqrt{2})^{l-1}\sqrt{\scal_k(p_k)}}\\
          &=:C_k \leq \infty,
        \end{split}
      \end{equation*}
      where $C_k$ is a constant independent of $i$. 

      Therefore, $(x_i,\tau_i)$ subconverges to a limit, say
      $(x_\infty,\tau_\infty)$. Then, by continuity of scalar curvature,  
      $\scal_{\tau_\infty}(x_\infty)= \infty$, which is absurd.
    \end{proof}
    
    Now we consider the volume ratios. 
    By Bishop-Gromov relative volume comparison theorem, for any
    $0< r\leq \ee $,
    \begin{equation*}
      \frac {\vol_{\tbar_k}(\pbar_k,r)}{r^n}
      =\frac {\vol_{\tbar_k}(\pbar_k,r)}{\vbh r}\frac{\vbh r}{r^n}
      \geq \frac {\vol_{\tbar_k}(\pbar_k,\ee)}{\vbh{\ee}}\frac{\vbh r}{r^n}
      \geq \frac{v'}{\vbh{\ee}}\frac{\vbh r}{r^n}.
    \end{equation*}
    Using relative volume comparison theorem again for euclidean and
    hyperbolic space, for any $0<r$,
    \begin{equation*}
      c(n) \geq \frac{r^n}{\vbh r}
    \end{equation*}
    where $c(n)$ is a constant depending only on dimension $n$. So 
    \begin{equation*}
      \frac {\vol_{\tbar_k}(\pbar_k,r)}{r^n}      
      \geq \frac{v'}{\vbh e}\frac 1 {c(n)}=:v''>0
    \end{equation*}
    where $v''=v''(v_0,n)$ is a constant depending only on $v_0$ and
    $n$. 

    Let $Q_k=\scal_{\tbar_k}(\pbar_k)$.
    We rescale metric $g_{\tbar_k}$ on
    $B_{\tbar_k}(\pbar_k,k/\sqrt{Q_k})$ for $k$ large. In this case,
    $k/\sqrt{Q_k} < e$, so above 
    lower bound for volume ratio is true for $B_{\tbar_k}(\pbar_k,r)$ with
    any $r<k/\sqrt{Q_k}$.   We define a new metric by 
    parabolic rescaling
    \begin{equation*}
      \gtil_k(x,s):=Q_kg(x,\tbar_k + Q_k^{-1}s).
    \end{equation*}
    Hereafter $\widetilde{\scal}_k$ is the scalar curvature of
    $\gtil_k$, so are the other $\tilde{\phantom{g}}$ed quantities. 
    Then by definition, $\gtil_k$ is a solution to ricci flow on
    $(-k,0]\times B_{\gtil_k(0)}(\pbar_k,k)$. 
    Note by equation~\eqref{eq:8}, scalar curvature of $\gtil_k$ is
    bounded from above by 2. And the complex sectional curvature
    \begin{equation}\label{eq:9}
      \widetilde{\sec}^c\geq -\frac 1{Q_k} > -1.
    \end{equation}
    Consequently, by Lemma \ref{lem:3} and  \ref{lem:4}, for some
    constant $c(n)>0$ depending only on dimension,
    \begin{equation}\label{eq:5}
      c(n)|\widetilde{\riem}|\leq \widetilde{\scal}\leq 2.
    \end{equation}
    Because volume ratio is scaling-invariant,
    \begin{equation}\label{eq:6}
      \frac{\vol_{\gtil_k(0)}(B_{\gtil_k(0)}(\pbar_k,r))}{r^n}\geq v'' > 0
    \end{equation}
    for any  $0< r\leq k$. Therefore, combining
    equation \eqref{eq:5} and \eqref{eq:6} with injectivity radius
    estimate of Cheeger-Gromov-Taylor\cite{MR658471} (see also
    Cabezas-Rivas and Wilking \cite{2011arXiv1107.0606C} Theorem C.3), we have
    \begin{equation*}
      \inj_{\gtil(0)}(\pbar_k) \geq c(v'',n) >0.
    \end{equation*}
    Together with equation \eqref{eq:5},  applying Hamilton's compactness
    theorem (see  
    \cite{MR1333936} also \cite{MR2302600}), the sequence
    of pointed ricci flow solution
    $(B_{\gtil_k}(\pbar_k,k),\gtil_k(s),\pbar_k)$ subconverge
    to a complete ancient solution 
    $(M_\infty,g_\infty(s),p_\infty)$ for $s\in (-\infty,0]$ in the
    pointed Cheeger-Gromov sense. From the convergence and
    $\widetilde{\scal}_k(\pbar_k)=1$, 
    $\scal_{g_\infty(0)}(p_\infty)=1$, hence the
    limit metric is non-flat.  For the
    same reason, its riemannian curvature is
    bounded from equation~\eqref{eq:5} and its complex sectional
    curvature is nonnegative from equation
    \eqref{eq:9}. It is also noncompact for every $s\in
    (-\infty,0]$. We first observe the diameter of $g_\infty(0)$ 
    is infinity. And we pick a sequence of points $y_i\in M_\infty$
    for $i=1,2,...$,
    s.t. $\dist_{g_\infty(0)}(p_\infty,y_i) = i$. Because ricci
    curvature is nonnegative, using Lemma
    \ref{lem:5} again, we show
    $\dist_{g_\infty(-s)}(p_\infty,y_i)\geq i \to \infty$. 
    With above observations, we can apply Lemma 4.5 in Cabezas-Rivas
    and Wilking \cite{2011arXiv1107.0606C} to this ancient solution,
    and get, for any $s\in (-\infty, 0]$,
    \begin{equation*}
      \lim_{r\to\infty} \frac {\vol_{g_\infty(s)}(B_{g_\infty}(\cdot,r))}{r^n}=0. 
    \end{equation*}
    However, this contradicts to the fact that, for any $r>0$,
    \begin{equation*}
      \frac {\vol_{g_\infty(0)}(B_{g_\infty(0)}(p_\infty,r))}{r^n}>0
    \end{equation*}
     from equation \eqref{eq:6}.
  \end{proof}
  Now we are ready to use following theorem to conclude the proof.
  \begin{thm}[B.L.Chen\cite{MR2520796} or M.Simon\cite{MR2439551}]
    \label{thm:5}
    There is a constant $C=C(n)$ with the following property. Suppose
    we have a smooth solution to the ricci flow on $M^n\times [0,T]$
    such that $B_t(x_0,r_0)$, $0\leq t\leq T$, is compactly contained
    in $M$ and
    \begin{enumerate}
    \item $|\riem|\leq r_0^{-2}$ on $B_0(x_0,r_0)$ at $t=0$;
    \item
      \begin{equation*}
        |\riem|_t(x)\leq \frac K t
      \end{equation*}
      where $K\geq 1$, $\dist_t(x_0,x)<r_0$, whenever $0\leq t\leq T$. 
    \end{enumerate}
    Then we have
    \begin{equation*}
      |\riem|_t(x)\leq \e{CK}(r_0-\dist_t(x_0,x))^{-2}
    \end{equation*}
    whenever $0\leq t \leq T$, $\dist_t(x_0,x)<r_0$.
  \end{thm}
  Let $x_0$ be any point in $M$. By our assumption, the ricci flow is
  complete for any $t$ hence any 
  ball in $M$ is compactly contained. Again by assumption,
  $|\riem_0|_0$ is bounded. 
  Thus we may choose $r_0$ small enough so that condition 1 is
  satisfied. Finally, lemma \ref{lem:2} gives condition 2, i.e.
  \begin{equation*}
   c(n)|\riem|_t(x)\leq \scal_t\leq \frac C t.
  \end{equation*}
  Consequently, there exists a $C$ depending on only dimension $n$,
  such that 
  \begin{equation*}
    |\riem_t(x)|\leq C
  \end{equation*}
  whenever $0\leq t\leq \epsilon$ and $\dist_t(x_0,x)\leq
  \frac{C'}{r_0^2}$. For $x_0\in M$ is arbitrary, we have $\riem$ is
  uniformly bounded on $M\times [0,\epsilon]$.
\end{proof}

\section{Corollaries}
\label{sec:corollaries}

We conclude the paper by a few corollaries of the main theorem. 

\subsection{Solutions with  Nonnegative Curvature}
\label{sec:case-nonn-curv}

If we assume further, that $(M^n,g)$ has nonnegative complex sectional
curvature, together with upper bound for scalar curvature,
it has nonnegative sectional curvature by Lemma \ref{lem:3} and
\ref{lem:4} stated in the next section. 
Then $g$ has a positive lower bound for injectivity radius
by using Toponogov's injectivity radius estimate as follows.
\begin{thm}[Toponogov\cite{MR0262974} also Maeda\cite{MR0365411}]
  \label{thm:2}
  If $M$ is a complete noncompact riemannian manifold and for all
  tangent two-plane $\sigma$ its sectional curvature $K_\sigma$
  satisfies the inequality $0\leq K(\sigma)\leq \lambda$ then there
  exists a constant $i_0>0$ such that for all $p\in M$the injectivity
  radius $i(p)$ satisfies 
  \begin{equation*}
    i(p)\geq i_0.
  \end{equation*}
  Further, if $0< K(\sigma)\leq \lambda$ for all $\sigma$, then for all
  $p\in M$, 
  \begin{equation*}
    i(p)\geq \frac \pi {\sqrt{\lambda}}.
  \end{equation*}
\end{thm}
\begin{rmk}
  \label{rmk:2}
  Beware the lower bound $i_0$ in the
  first part, 
  depends not only $\lambda$ but also on $g$ in a general way. In fact,
  a sequence of 
  ``thiner'' and  ``thiner'' flat tori whose injetivity radii approach
   $0$ is an example.  We only use this part of the theorem in this paper.
\end{rmk}
Because the sectional
curvature is bounded from above, say by a positive constant $C_0$, by 
volume comparison, the volume of any ball with fixed radius
$i_0$ a lower bound. Rescaling the metric
if necessary, by Theorem~\ref{thm:3}, we get the 
following corollary.
\begin{cor}
  \label{cor:2}
  Let $(M^n,g)$ be a complete noncompact $n$-dimensional  riemannian
  manifold, with complex sectional curvature $\sec^c \geq 0$
  and $\scal\leq s_0$ for some $s_0 > 0$. Let $g_i(t)$ $i=1,2$ be 
  two ricci flow solutions on
  $M\times [0,T]$, with same initial data $g_i(0) = g$. If both of them are
  complete riemannian metrics and
  $\sec^c(g_i(t))\geq 0$ for every $t\in [0,T]$, 
  then there exists a constant $\epsilon>0$, s.t.
  $g_1(t) = g_2(t)$ on $M\times [0,\min\{\epsilon,T\}]$.
\end{cor}

\subsection{Solutions on Three Manifolds}
\label{sec:case-three-manifold}

Next let us consider the three manifold. In this case sectional curvature,
complex sectional curvature, and curvature operator 
bounded from below are all equivalent. Then we have the following corollary.
\begin{cor}
  \label{cor:3}
    Let $(M,g)$ be a smooth complete riemannian three manifold with bounded sectional
  curvature. And the volume of unit balls in $(M,g)$ is uniformly
  bounded by $v_0>0$. Let $g_i(t)$, $i=1,2$
  $t\in[0,T]$ be two complete solutions of ricci flow with $g_i(0)=g$. If their
  sectional curvatures are both bounded from below for any 
  $t\in [0,T]$, then there exists a constant $C_n$ depending on dimension $n$, s.t.
  $g_1(t) = g_2(t)$ on $M\times [0,\min\{\epsilon,T\}]$, where
  $\epsilon=\min\{v_0/2C_n,1/(n-1)\}$.
\end{cor}
\begin{rmk}
  \label{rmk:8}
  Readers may notice that above corollary together with the
  fact that nonnegativity of sectional curvature is preserved by
  $3$-dimensional along ricci flow(Corollary 2.3 in
  \cite{MR2520796}), imply B-L Chen's strong uniqueness
  theorem (Theorem 1.1 in same paper above). However, this in
  essential is not a different proof because two crucial 
  components in it are from B-L Chen's paper. One is the
  preservation of nonnegative sectional curvature, the other one is
  Theorem~\ref{thm:5} which we use to prove our main theorem. 
\end{rmk}

\subsection{Warped Product Solution}
\label{sec:case-warped-product}

If we consider the ricci flow solutions in only the
class of rotationally symmetric metrics, the curvature condition can
be weaken to ricci bounded from below. Actually we can slightly
enlarge the class. For convenience of stating our corollary, we need the
following definition. 
\begin{definition}
  \label{def:3}
  We call $g$ a warped product metric on a manifold $M$ if they
  satisfy the following conditions:
  \begin{enumerate}
  \item $M$ is diffeomorphic to $I\times N$ for some interval $I$ and
    manifold $N$.
  \item There exist a smooth function $f:I\mapsto\rn^+$ and a
    riemannian metric $h$ on $N$ s.t. $g=\dd r^2+f^2(r)h$ on coordinate
    system $(r,p)\in I\times N$.
  \end{enumerate}
  We call $(N,h)$ the warped component and $\D r$ the radial direction.
\end{definition}
\begin{cor}
  \label{cor:4}
  Let $(M,g)$ be a smooth complete noncompact riemannian manifold 
  satisfying the following conditions:
  \begin{enumerate}
  \item $(M,g)$ has bounded sectional curvature.
  \item Volume of any unit ball in $(M,g)$ is bounded from below uniformly
    by a positive constant.
  \item $g$ is a warped product metric and its warped component has nonnegative
    curvature operator.
  \end{enumerate}
  Let $g_i(t)$, $i=1,2$ $t\in[0,T]$ be two complete
  solutions of ricci flow with $g_i(0)=g$. In addition, for any fixed
  $t\in [0,T]$, they satisfy the following conditions:
  \begin{enumerate}
  \item $g_i(t)$ is a warped product metric whose warped
    component is complete and has nonnegative curvature operator.
  \item  Ricci curvature of $g_i$ is bounded from below in
    radial direction by a uniform constant independent of $t$. 
  \end{enumerate}
  Then there exists an $\epsilon>0$ such that $g_1(t)=g_2(t)$
  for any $t\in [0,\min\{\epsilon,T\}]$. 
\end{cor}

Above corollary is a consequence of facts below. 
For convenience, we introduce $\phi=\log f$. Then
$f''/f=\phi''+(\phi')^2$ and $f'/f=\phi'$. Because $f>0$ is a smooth
function, $\phi$ is smooth wherever $f$ is defined. 
Then the curvatures (see e.g. Petersen\cite{MR2243772})
\begin{equation*}
  \R(\D r\wedge X)=-(\phi''+(\phi)^2)\D r\wedge X.
\end{equation*}
where $X$ is any vector field on $N$.
Let $\{E_a\}$, $1\leq a \leq n(n-1)/2$ be a o.n. frame of
$\wedge^2T_xN$, which diagonalize the curvature operator $\R^h$,
with correspondent eigenvalues $\lambda_a$. It is easy to see
\begin{equation*}
  \R(E_a)=(\lambda_a\e{-2\phi}-(\phi')^2)E_a.
\end{equation*} 
Let $e_\alpha$ $0\leq\alpha\leq n+1$ be an o.n. frame for $g$ and
$e_0=\D r$. Hence $e_i$ $i=1,2,\cdots,n$ are orthogonal frame on $(N,h)$.
We can choose $e_i$ further s.t. they diagonalize $\ric^h$, then for
any $i\neq j$,
\begin{equation*}
  \ric(e_0,e_0)=-n(\phi''+(\phi')^2)\quad \ric(e_0,e_i)=0=\ric(e_i,e_j).
\end{equation*}

Then let us prove the following lemma.
\begin{lemma}
  \label{lem:6}
  Let $a:[0,+\infty)\to \rn$ be a smooth function s.t. $a'+a^2\leq C^2$
  for some constant $C\geq 0$. Then, when $C>0$,
  $-C\leq a\leq \max\{a(0),C\}$, when $C=0$, $0\leq a\leq a(0)$.
\end{lemma}
\begin{proof}
  The case $C>0$. Consider the correspondent ODE
  \begin{equation*}
    A'+A^2=C^2.
  \end{equation*}
  When $A(0)=\pm C$, $A\equiv \pm C$.  
  When $A(0)\neq \pm C$, its solution is
  \begin{equation*}
    A=C+\frac{2C}{K\e{2Cr}-1}\quad
    \text{where $K$ is a constant satisfying}\ A(0)=C(1+\frac 2 {K-1}).
  \end{equation*}
  Then let us has a closer look of $A$ with different initial values
  other than $\pm C$.
  \begin{enumerate}
  \item   When $A(0)> C$. Then $K > 1$. Hence  $A(r)$ is 
    decreasing and $A\to C$ as $r\to\infty$. So $A(0)\geq A\geq C$ as
    long as $r\geq 0$. 
  \item   When $-C<A(0)<C$. Then $K<0$. So $A(r)$ is increasing and $A\to C$
    as $r\to \infty$. So   $C\geq A\geq -C$ as long as $r\geq 0$.
  \item When $A(0)< -C$.  Then $0<K<1$. So $A(r)$ is decreasing and $A\to
    -\infty$ as $r\to r_0$ for some $0<r_0< \infty$.
  \end{enumerate}
  
  Let $A(0)=a(0)$, by comparison principle of ordinary differential
  equations, $a(r)\leq A(r)$ as long as $A(r)$ exists. So,
  when $a(0)<-C$, $a(r)$ becomes discontinuous at some finite $r$,
  which contradicts to our assumption on $a$. Consequently, $a(0)\geq
  -C$,
  thus $a(r)\leq A(r)\leq \max\{a(0),C\}$. Not only so, we can 
  further conclude that $a(r)\geq -C$. Otherwise, there exists some
  $\tilde{r}>0$ s.t. $a(\tilde{r})< -C$. Due to the fact that the
  differential inequality 
  satisfied by $a$ is translation invariant, it becomes discontinuous
  for the same reason mention as above. 

  Conclusion for  $C=0$ follows a similar argument. 
\end{proof}

We have the following fact for warped product metrics.
\begin{prop}
  \label{prop:2}
     Let $(N,h)$ be a complete riemannian $n$-manifold with nonnegative
   curvature operator. Let $g=\dd r^2 + f^2(r)h$ be a complete
   noncompact metric on $\rn\times N$ or $\rn^+\times N$. If the ricci
   curvature of $g$ $\ric(\D r, \D r)\geq -C^2$ for some constant $C\geq
   0$, then the  curvature operator $\R\geq -C'\id$ for constant
   $C'$. 
\end{prop}
\begin{rmk}
  \label{rmk:7}
  When $f(r)$ is defined on  $\rn^+\times N$, $g$ is complete if and
  only if
  $f(r)$ is with certain restrictions and  $(N,h)$ is 
  a standard sphere (up to scaling). Details can be found in
  Petersen\cite{Petersen-warped-product-e}.
\end{rmk}
\begin{proof}
  Without loss of generality, assume $f>0$ is a smooth function on
  $[0,+\infty)$. 
  Since $\ric(\D r,\D r)\geq -C^2$, 
  \begin{equation*}
    \phi''+(\phi')^2\leq \frac{C^2}{n}.
  \end{equation*}
  By Lemma \ref{lem:6}, $\phi'$ is bounded. 
  And by assumption $\R^h$ is nonnegative, $\R(E_a)\geq
  -(\phi')^2E_a$ is bounded from below.
\end{proof}
\begin{proof}[Proof of Corollary~\ref{cor:4} ]
  By lemma above, we get $g_i(t)$ are both with curvature
  operators bounded from below,  hence with complex sectional
  curvature bounded from below. Then the
  corollary follows from Theorem~\ref{thm:1}.
\end{proof}

\noindent\textsc{acknowledgement: } Both authors thank professor Xiu-xiong
Chen for his valuable suggestions, constant motivating, and
encouragement. The first author thanks professor Xiu-xiong Chen and 
Mathematics Department of Stony Brook University for their hospitality
during his visit and preparation of this paper. 



\begin{thebibliography}{10}

\bibitem{MR2449060}
Simon Brendle and Richard Schoen.
\newblock Manifolds with {$1/4$}-pinched curvature are space forms.
\newblock {\em J. Amer. Math. Soc.}, 22(1):287--307, 2009.

\bibitem{2011arXiv1107.0606C}
E.~{Cabezas-Rivas} and B.~{Wilking}.
\newblock {How to produce a Ricci Flow via Cheeger-Gromoll exhaustion}.
\newblock {\em ArXiv e-prints}, July 2011.

\bibitem{MR2394158}
Jeff Cheeger and David~G. Ebin.
\newblock {\em Comparison theorems in {R}iemannian geometry}.
\newblock AMS Chelsea Publishing, Providence, RI, 2008.
\newblock Revised reprint of the 1975 original.

\bibitem{MR658471}
Jeff Cheeger, Mikhail Gromov, and Michael Taylor.
\newblock Finite propagation speed, kernel estimates for functions of the
  {L}aplace operator, and the geometry of complete {R}iemannian manifolds.
\newblock {\em J. Differential Geom.}, 17(1):15--53, 1982.

\bibitem{MR2520796}
Bing-Long Chen.
\newblock Strong uniqueness of the {R}icci flow.
\newblock {\em J. Differential Geom.}, 82(2):363--382, 2009.

\bibitem{MR2260930}
Bing-Long Chen and Xi-Ping Zhu.
\newblock Uniqueness of the {R}icci flow on complete noncompact manifolds.
\newblock {\em J. Differential Geom.}, 74(1):119--154, 2006.

\bibitem{MR2721664}
Qing Chen and Yajun Yan.
\newblock Uniqueness for {R}icci flow with unbounded curvature in dimension 2.
\newblock {\em Ann. Global Anal. Geom.}, 38(3):293--303, 2010.

\bibitem{MR2302600}
Bennett Chow, Sun-Chin Chu, David Glickenstein, Christine Guenther, James
  Isenberg, Tom Ivey, Dan Knopf, Peng Lu, Feng Luo, and Lei Ni.
\newblock {\em The {R}icci flow: techniques and applications. {P}art {I}},
  volume 135 of {\em Mathematical Surveys and Monographs}.
\newblock American Mathematical Society, Providence, RI, 2007.
\newblock Geometric aspects.

\bibitem{MR697987}
Dennis~M. DeTurck.
\newblock Deforming metrics in the direction of their {R}icci tensors.
\newblock {\em J. Differential Geom.}, 18(1):157--162, 1983.

\bibitem{2012arXiv1210.0813G}
P.~{Gianniotis}.
\newblock {The Ricci flow on manifolds with boundary}.
\newblock {\em ArXiv e-prints}, October 2012.

\bibitem{MR2832165}
Gregor Giesen and Peter~M. Topping.
\newblock Existence of {R}icci flows of incomplete surfaces.
\newblock {\em Comm. Partial Differential Equations}, 36(10):1860--1880, 2011.

\bibitem{0504.53034}
Richard~S. Hamilton.
\newblock {Three-manifolds with positive Ricci curvature.}
\newblock {\em J. Differ. Geom.}, 17:255--306, 1982.

\bibitem{MR1333936}
Richard~S. Hamilton.
\newblock A compactness property for solutions of the {R}icci flow.
\newblock {\em Amer. J. Math.}, 117(3):545--572, 1995.

\bibitem{MR1375255}
Richard~S. Hamilton.
\newblock The formation of singularities in the {R}icci flow.
\newblock {\em {Surveys in differential geometry, {V}ol.\ {II} }}, pages
  7--136, 1995.

\bibitem{MR831655}
Fritz John.
\newblock {\em Partial differential equations}, volume~1 of {\em Applied
  Mathematical Sciences}.
\newblock Springer-Verlag, New York, fourth edition, 1982.

\bibitem{2012arXiv1206.3225K}
B.~{Kotschwar}.
\newblock {An energy approach to the problem of uniqueness for the Ricci flow}.
\newblock {\em ArXiv e-prints}, June 2012.

\bibitem{MR0365411}
Masao Maeda.
\newblock On the injective radius of noncompact {R}iemannian manifolds.
\newblock {\em Proc. Japan Acad.}, 50:148--151, 1974.

\bibitem{2002math.....11159P}
G.~{Perelman}.
\newblock {The entropy formula for the Ricci flow and its geometric
  applications}.
\newblock {\em ArXiv Mathematics e-prints}, November 2002.

\bibitem{2003math......3109P}
G.~{Perelman}.
\newblock {Ricci flow with surgery on three-manifolds}.
\newblock {\em ArXiv Mathematics e-prints}, March 2003.

\bibitem{Petersen-warped-product-e}
Peter Petersen.
\newblock Warped product.
\newblock \url{http://www.math.ucla.edu/~petersen/warpedproducts.pdf}.
\newblock [Online; accessed 5 Oct 2013].

\bibitem{MR2243772}
Peter Petersen.
\newblock {\em Riemannian geometry}, volume 171 of {\em Graduate Texts in
  Mathematics}.
\newblock Springer, New York, second edition, 2006.

\bibitem{MR2423998}
A.~M. Petrunin.
\newblock An upper bound for the curvature integral.
\newblock {\em Algebra i Analiz}, 20(2):134--148, 2008.

\bibitem{MR1001277}
Wan-Xiong Shi.
\newblock Deforming the metric on complete {R}iemannian manifolds.
\newblock {\em J. Differential Geom.}, 30(1):223--301, 1989.

\bibitem{MR2439551}
Miles Simon.
\newblock Local results for flows whose speed or height is bounded by {$c/t$}.
\newblock {\em Int. Math. Res. Not. IMRN}, pages Art. ID rnn 097, 14, 2008.

\bibitem{MR0262974}
V.~A. Toponogov.
\newblock Theorems on minimizing paths in noncompact {R}iemannian spaces of
  positive curvature.
\newblock {\em Dokl. Akad. Nauk SSSR}, 191:537--539, 1970.

\bibitem{2013arXiv1305.1905T}
P.~M. {Topping}.
\newblock {Uniqueness of Instantaneously Complete Ricci flows}.
\newblock {\em ArXiv e-prints}, May 2013.

\end{thebibliography}

\end{document}